\documentclass[11pt]{article}
\usepackage{layout}
\usepackage{amssymb}
\usepackage{epsfig}
\usepackage{graphicx}
\usepackage{color}
\usepackage{latexsym}
 \usepackage{caption}
\usepackage{subcaption}
\usepackage{algorithmicx}
\usepackage[ruled]{algorithm}
\usepackage{algpseudocode}
\usepackage{algpascal}
\usepackage{algc}

\alglanguage{pseudocode}

\newcommand{\halmos}{\hfill \ifhmode\unskip\nobreak\fi\ifmmode\ifinner\else\hskip5pt
 \fi\fi \hbox{\hskip5pt\vrule width4pt height6pt depth1.5pt\hskip5pt}}

\newenvironment{proof}[1][Proof]{\textbf{#1.} }{\ \rule{0.5em}{0.5em}}
\newcommand{\Frac}[2] {\frac{\textstyle #1} {\textstyle #2}}
\newcommand{\R}{I\!\!R}

\newtheorem{theorem}{Theorem}[section]

\newtheorem{remark}{Remark}[section]
\newtheorem{proposition}{Proposition}[section]

\begin{document}

\title{Geometrical inverse matrix approximation for least-squares problems and acceleration strategies}

 \author{Jean-Paul Chehab\thanks{
LAMFA, {\small UMR} 7352,
 Universit\'e de Picardie Jules Verne, 33 rue Saint Leu, 80039 Amiens France({\tt
 Jean-Paul.Chehab@u-picardie.fr})}
  \and
  Marcos Raydan \thanks{Departamento de C\'omputo Cient\'{\i}fico y Estad\'{\i}stica, Universidad Sim\'on Bol\'{\i}var,
           Ap. 89000, Caracas 1080-A, Venezuela ({\tt mraydan@usb.ve})}
}

\date{February 21, 2019}

\maketitle

\begin{abstract}
 We extend the geometrical inverse approximation approach for solving linear
 least-squares problems. For that we focus on the minimization of $1-\cos(X(A^TA),I)$, where $A$ is a given rectangular
  coefficient matrix and $X$ is the approximate inverse. In particular, we adapt the recently published simplified gradient-type iterative
   scheme MinCos to the least-squares scenario. In addition,  we combine the generated convergent
    sequence of matrices  with well-known acceleration strategies based on recently developed matrix extrapolation methods, and also with
     some deterministic and  heuristic acceleration schemes  which are based
    on affecting, in a convenient way, the steplength at  each iteration. A set of numerical experiments, including large-scale problems,
     are presented to illustrate the performance of the different accelerations strategies.      \\ [2mm]
 {\bf Key words:} Inverse approximation, cones of matrices, matrix  acceleration techniques, gradient-type methods.
\end{abstract}

\section{Introduction}

The development of  inverse matrix approximation strategies for solving linear least-squares problems is an active research area since they play a
key role in a wide variety of  science and engineering applications involving  ill-conditioned large  matrices (sparse or dense); see e.g.,
  \cite{carrbg, Chehab, Chehab16, Chen01, ChowSaad97, Chung, forsman,  helsing, gonzalez2013, sajo, sidje, wang}.

  In this work, for a given real  rectangular  $m\times n$ ($m\geq n$)  matrix $A$,  we  will obtain inverse approximations based on minimizing the positive-scaling-invariant function $\widehat{F}(X) = 1-\cos(X(A^TA),I)$ on a suitable closed and bounded subset of the  cone of
  symmetric and positive semidefinite matrices ($PSD$). Therefore, our inverse approximations will remain in the $PSD$ cone, in sharp contrast with
   the standard approach of minimizing  the Frobenius norm of the residual $(I - X(A^TA))$, for which a symmetric and positive
    definite approximation cannot be guaranteed; see e.g., \cite{Cui, gouldscott}.

  For the minimization of $\widehat{F}(X)$ we will extend  and adapt the simplified gradient-type scheme MinCos, introduced in \cite{ChehabRaydan15},
   to the linear least-squares scenario. Moreover, we will adapt and apply some well-known modern matrix acceleration strategies to the generated
   convergent sequences.
   In particular, we will focus our attention on the use of the simplified topological $\varepsilon$-algorithms
  \cite{BrezinskiZ14, BrezinskiZ17},  and  also on the extension of  randomly-chosen steplength acceleration strategies \cite{RaydanSvaiter}, as
  well as the extension of some recent nonmonotone gradient-type choices of steplenghts \cite{frasso, zhou}.

  The rest of the document is organized as follows. In Section \ref{smincos}, we recall the MinCos method for matrices in the $PSD$ cone,
  and briefly describe its most important properties. In Section \ref{smincosls}, we develop the extended and adapted version of the MinCos
   method for solving linear least-squares problems. In Section \ref{saccel}, we describe the different adapted acceleration strategies
    to hopefully observe a faster convergence of the generated sequences. In Section \ref{numres}, we present experimental numerical results
  to illustrate the performance of the adapted algorithm for least-squares problems, and also to illustrate the advantages of using acceleration techniques
  over a set of problems, including large-scale matrices.

\section{The MinCos method} \label{smincos}

Let us recall the MinCos method for the minimization of $F(X) = 1-\cos(XA,I)$, when  $A$ is an $n\times n$ real symmetric and positive definite matrix.
\begin{center}
\begin{minipage}[H]{14.5cm}
  \begin{algorithm}[H]
    \caption{: MinCos (simplified gradient approach on $F(X)=1-\cos(XA,I)$)} \label{mincos}
    \begin{algorithmic}[1]
        \State Given $X^{(0)}\in PSD$ (commuting with $A$)
       \For{$k=0,1, \cdots$ until a stopping criterion is satisfied,}
          \State {\bf Set} $w_k = \langle X^{(k)}A, I\rangle$
          \State {\bf Set}  $\widehat{D}_k = - \Frac{1}{n}\left(\Frac{w_k}{n}X^{(k)}A-I\right)$
          \State {\bf Set} $\alpha_k = \left\vert\Frac{n\:\langle  \widehat{D}_kA, I\rangle - w_k \langle X^{(k)}A,  \widehat{D}_k A\rangle}
               {\langle \widehat{D}_k A, I\rangle \langle X^{(k)}A, \widehat{D}_k A\rangle - w_k \|\widehat{D}_k A\|_F^2}\right\vert$
           \State {\bf Set}$Z^{(k+1)}=X^{(k)} + \alpha_k \widehat{D}_k$
           \State {\bf Set} $X^{(k+1)}=s \sqrt{n} \Frac{Z^{(k+1)}}{\|Z^{(k+1)}A\|_F}$, where $s=1$ if $trace(Z^{(k+1)}A)>0$, $s=-1$ else
       \EndFor
    \end{algorithmic}
    \end{algorithm}
\end{minipage}
\end{center}
This method has been successfully introduced in \cite{ChehabRaydan15}, and can be seen as  an improved version of the Cauchy Method applied to
 the minimization of the merit function
$$
F(X)=1-cos(XA,I)=1-\Frac{\langle XA,I\rangle}{\|XA\|_F\|I\|_F},
$$
where $I$ is the identity matrix,  $\langle A,B \rangle =trace(A^TB)$ is the Frobenius inner product
 in the space of matrices and  $\|\:.\:\|_F$ is the associated Frobenius norm. In here PSD refers to the positive semi-definite closed cone of
 square matrices which possesses a rich geometrical structure; see, e.g., \cite{andraytar, ChehabRaydan}.

 In Remark \ref{remkmcos} we summarize the most important properties of Algorithm \ref{smincos} (see  \cite{ChehabRaydan15}).
\begin{remark} \label{remkmcos}
\begin{enumerate}
\item  The minimum of $F(X)$ is reached at $X$ such that $AX=\alpha I$.
But if we impose $\|AX\|_F=\|I\|_F= \sqrt{n}$, we have $\alpha=\pm 1$. If in addition we impose $trace(XA)\geq 0$, we have $XA=I$.
\item By  construction $\|X^{(k)}A\|_F=\sqrt{n}$, for all $k \geq 1$. If we choose $X^{(0)}$ such that
$trace(X^{(0)}A) = \langle X^{(0)}A,I \rangle > 0$ then  by construction all the iterates remain in the $PSD$ cone.
Moreover, if in addition  $X^{(0)}A=AX^{(0)}$, then $X^{(k)}A=AX^{(k)}$ and  $Z^{(k)}A=AZ^{(k)}$, for all $k \ge 0$.
 \item Unless we are at the solution, the search direction  $\widehat{D}_k$ is
  a descent direction for the function $F$ at $X$. The steplength  $\alpha_k>0$ is  the optimal choice,  i.e., the positive steplength that
  (exactly) minimizes the function $F(X)$ along the  direction $\widehat{D}_k$.  Furthermore,  $Z^{(k)}$, $X^{(k)}$, and $X^{(k)}A$ in the
  MinCos Algorithm  are symmetric matrices for all $k$, which  are also uniformly bounded away from zero, and so the algorithm is well-defined.
  \end{enumerate}
 \end{remark}
   For completeness, we state the convergence result concerning the MinCos method (for the proof see \cite{ChehabRaydan15}).
     \begin{theorem} \label{teocnv}
 The sequence $\{X^{(k)}\}$ generated by the MinCos Algorithm converges to $A^{-1}$.
  \end{theorem}

\section{The MinCos method for least-squares problems} \label{smincosls}

Let us now consider linear systems involving the real rectangular $m\times n$ ($m\geq n$) matrix $A$, for which solutions do not exist. An interesting
 and always robust available option is to use the least-squares approach, i.e., to solve instead the normal
  equations, which involve solving a linear system with the square matrix $A^TA$ that belongs to the $PSD$ cone. Let us  assume that
  $A$ is full column rank, i.e., that $A^TA$ is symmetric and positive definite. In that case, it is always an available (default) option, although
   not recommendable,  to apply Algorithm \ref{mincos} directly on the matrix $A^TA$. Nevertheless, to avoid multiplications with the matrix
   $A^T$ (which is usually not available for practical applications), and also to avoid unnecessary and numerically risky calculations, we will
    adapt each one of the steps of the MinCos algorithm.  For that we first need to recall  that, using properties of the trace operator, for any
    matrices $W_1$, $W_2$, and $W_3$ with the proper sizes, it follows that
 \begin{equation} \label{trprop}
  \langle W_1, W_2 W_3 \rangle_F = \langle W_2^TW_1, W_3 \rangle_F = \langle W_1W_3^T,W_2  \rangle_F.
 \end{equation}
   For any given  matrix $Y$ for which $Y^T$ is available, using (\ref{trprop}), we obtain that
 \begin{equation} \label{trick}
 \langle Y A^TA, I\rangle = \langle (AY^T)^TA, I\rangle = \langle A, AY^T\rangle  = \langle AY^T, A\rangle.
 \end{equation}
  Hence, using (\ref{trick}) and the fact that $\widehat{D}_k$ is symmetric,  it follows that
  $\langle  \widehat{D}_kA^TA, I\rangle =  \langle  A\widehat{D}_k, A\rangle$, and
   since $X^{(k)}$ is symmetric, $\langle  X^{(k)}A^TA, I\rangle =  \langle  AX^{(k)}, A\rangle$.   Moreover,
\[  \langle X^{(k)}A^TA, \widehat{D}_k A^TA\rangle  =  \langle (AX^{(k)})^TA, (A\widehat{D}_k)^TA\rangle. \]
  Similarly, we obtain that   $\|Y A^TA\|_F^2 = \|(A Y^T)^T A\|_F^2$.
  Summing up we obtain the following extended version of the MinCos algorithm  for minimizing $\widehat{F}(X) = 1-\cos(X(A^TA),I)$, where $A$
 is a given rectangular matrix.
 \begin{center}
\begin{minipage}[H]{14.5cm}
  \begin{algorithm}[H]
    \caption{: MinCos for minimizing  $\widehat{F}(X) = 1-\cos(X(A^TA),I)$} \label{mincosls}
    \begin{algorithmic}[1]
        \State Given $X^{(0)}\in PSD$ (commuting with $A^TA$)
       \For{$k=0,1, \cdots$ until a stopping criterion is satisfied,}
          \State {\bf Set} $C_k = AX^{(k)}\;$ and  $\;w_k = \langle  C_k, A\rangle$
          \State {\bf Set}  $\widehat{D}_k = - \Frac{1}{n}\left(\Frac{w_k}{n}C_k^TA-I\right)$, $\;\;B_k=A\widehat{D}_k,\;$ and
           $\;\mu_k = \langle  B_k^TA, C_k^TA\rangle$
          \State {\bf Set} $\beta_k = \langle  B_k, A\rangle\;$  and
           $\;\alpha_k = \left\vert\Frac{n\: \beta_k - w_k \mu_k}
            {\beta_k \mu_k - w_k \|B_k^T A\|_F^2}\right\vert$ 
           \State {\bf Set} $Z_{temp}=X^{(k)} + \alpha_k \widehat{D}_k\;$ and $\;Z^{(k+1)} = (Z_{temp} + Z_{temp}^T)/2$
           \State {\bf Set} $X^{(k+1)}=s  \Frac{\sqrt{n}\:Z^{(k+1)}}{\|(A Z^{(k+1)})^T A\|_F}$, where $s=1$ if $trace((A Z^{(k+1)})^T A)>0$, $s=-1$ else
       \EndFor
    \end{algorithmic}
    \end{algorithm}
\end{minipage}
\end{center}

 We note that for taking advantage of the adapted formulas obtained above, which are based mainly on (\ref{trick}), it is important to maintain the symmetry
 of all the iterates in Algorithm \ref{mincosls}. For that, let us observe that at Step 6,  $Z^{(k+1)}$ is actually obtained as the closest symmetric
  matrix to the original one given by $Z^{(k+1)} = X^{(k)} + \alpha_k \widehat{D}_k$. This additional calculation represents an irrelevant computational
   cost as compared to the rest of the steps in the algorithm, but at the same time it represents a safety procedure to avoid the numerical loss of symmetry
    that might occur  when $A$ is large-scale and ill-conditioned.
    We also note that for any matrix $A$, $A^TA$ is in the PSD cone, and so all the results presented
     in Section 2 apply for algorithm \ref{mincosls}, in particular since $A$ is full column rank then by Theorem \ref{teocnv} the sequence $\{X^{(k)}\}$
  converges to $(A^TA)^{-1}$.

\section{Acceleration strategies} \label{saccel}

The sequence of matrices  $\{X^{(k)}\}\subset \R^{n\times n}$ generated either by Algorithm \ref{mincos} or by Algorithm \ref{mincosls}
can be viewed as simplified and improved versions of the CauchyCos algorithm developed in \cite{ChehabRaydan15}, which is a specialized version
 of the Cauchy (steepest descent) method for minimizing $F(X)$. Nevertheless,  both algorithm are gradient-type methods,  and as a consequence they
 can be accelerated using some well-known  acceleration
  strategies  which extend  effective scalar and vector modern sequence acceleration techniques; see, e.g., \cite{Brezinski00, BrezinskiZ91}.
  We note that that the sequence  $\{X^{(k)}\}$,  generated by Algorithm \ref{mincos} or by Algorithm \ref{mincosls}, converges to the limit
   point $A^{-1}$ or $(A^TA)^{-1}$, respectively.

  For our first acceleration strategy, we will focus on the matrix version of the so-called simplified topological $\varepsilon$-algorithms,
  which belongs to the general family of acceleration schemes that transform the original sequence to produce a new one that hopefully
   will converge faster to the same limit point; see e.g., \cite{BrezinskiZ14, BrezinskiZ18, Graves00, jbilou16, jbilou00, jbilou15, matos92}.
   For a full historical review on this topic as well as some other related issues we recommend \cite{BrezinskiZ18b}.
  To use the simplified topological $\varepsilon$-algorithms, we take advantage of the  recently published  Matlab  package
  EPSfun\footnote{The Matlab  package EPSfun is freely available at http://www.netlib.org/numeralgo/}
  \cite{BrezinskiZ17}, that effectively implements the most advanced options of that family, including the matrix sequence versions.
  In particular, we focus on the matrix versions of the specific simplified topological $\varepsilon$-algorithms 1 (STEA1) and the
  simplified topological $\varepsilon$-algorithms 2 (STEA2) using the restarted option (RM), which have been implemented in the EPSfun package,
   including  four possible variants. The details of all the options that can be used in the EPSfun package are fully described in \cite{BrezinskiZ17}.

For our second acceleration strategy, we will adapt a procedure proposed and analyzed in \cite{RaydanSvaiter} for the minimization of convex quadratics,
 and that can be viewed as a member of the family for which the acceleration is generated by the method itself, i.e., at once and dynamically
  only one accelerated sequence is generated; see, e.g., \cite{Chehab16}. This approach can take advantage of the intrinsic characteristics of
   the method that generates the  original sequence, which in some cases has proved to produce more effective accelerations than the standard
   approach that transforms the original one to produce an independent accelerated one \cite{delay80}.  For our specific algorithms, the second
   strategy is obtained by relaxing the optimal descent parameter $\alpha_k$ as
$$
\alpha_k \leftarrow \theta_k \alpha_k,
$$
where $\theta_k$ is at each step randomly chosen in $(a,b)$, preferably $(a,b)=(1-\eta,1+\eta)$ for $0<\eta <1$, following a uniform distribution. In order
 to study the interval in which  $F(X^{(k+1)})$ attains a value less than or equal to $F(X^{(k)})$ (see \cite{RaydanSvaiter}), let us define
$$
\phi_k(t)=F(X^{(k)}+t\alpha_k\widehat{D}_k)-F(X^{(k)})
$$
where $\widehat{D}_k$ is the descent direction and $\alpha_k$ the optimal parameter given
 in Algorithm \ref{mincos}. Notice that all the results that follow for $F$ and Algorithm \ref{mincos} also applies automatically to
 $\widehat{F}$ and  Algorithm \ref{mincosls}.
\begin{proposition}
For all $k$, it holds that $\phi_k(0)=0$, $\phi_k'(0)<0$, and $\phi_k(1)<0$.
\end{proposition}
\begin{proof}
$\phi_k(0)=0$ by construction and $\phi'_k(0)<0$ since $\widehat{D}_k$ is a descent direction.
Now $\phi_k(1)<0$ because $\alpha_k$ minimizes $F(X^{(k)}+\alpha \widehat{D}_k)$.
\end{proof} \\
Our next result is concerned with the right  extreme value of the interval.
\begin{proposition}
If there exists $t^*_k>1$ such that $\phi_k(t^*_k)=0$, then we have
$$
t^*_k=\Frac{2\left(\langle X^{(k)}A,\widehat{D}_kA\rangle \langle X^{(k)}A,I\rangle^2 -n \langle \widehat{D}_kA,I\rangle\right)}
{\alpha_k\left( \|\widehat{D}_kA\|^2_F\langle X^{(k)}A,I\rangle^2-n\langle \widehat{D}_kA,I\rangle^2\right)}
$$
\end{proposition}
\begin{proof}
Forcing $\phi_k(t)=0$ implies that
\begin{eqnarray*}
 & & \|X^{(k)}A\|^2_F  \left(\langle X^{(k)}A,I\rangle + t\alpha_k\langle \widehat{D}_kA,I\rangle \right)^2 \\
 & = & \langle X^{(k)}A,I\rangle^2\left(\|X^{(k)}A\|_F^2+2\alpha_k t \langle X^{(k)}A,\widehat{D}_kA\rangle +\alpha_k^2 t^2
\|\widehat{D}_kA\|^2_F\right),
\end{eqnarray*}
 and we obtain
\begin{eqnarray*}
 & & t^2\alpha_k^2\left(\|X^{(k)}A\|^2_F\langle \widehat{D}_kA,I\rangle^2-\langle X^{(k)}A,I\rangle^2\|\widehat{D}_kA\|^2_F\right) + \\
 & + & 2t\alpha_k\left(\langle \widehat{D}_kA,I\rangle\|X^{(k)}A\|^2_F- \langle X^{(k)}A,\widehat{D}_kA\rangle \langle X^{(k)}A,I\rangle^2\right)=0.
\end{eqnarray*}
Now, dividing by $\alpha_k t \ne 0$ and using $\|X^{(k)}A\|_F=\sqrt{n}$, it follows that
$$
t\alpha_k \left(n\langle \widehat{D}_kA,I\rangle^2 -\langle X^{(k)}A,I\rangle^2\|\widehat{D}_kA\|^2_F\right)
=2\left(\langle \widehat{D}_kA,I\rangle n- \langle X^{(k)}A,\widehat{D}_kA\rangle \langle X^{(k)}A,I\rangle^2\right),
$$
and  the result is established.
\end{proof}
\begin{remark}
At each iteration $k$ we can compute $t^*_k$ when it exists and relax randomly $\alpha_k$
with $\theta_k$ uniformly randomly chosen in $(0,t^*_k)$. Notice that the computation of $t^*_k$ is obtained for free: all the terms defining $t^*_k$
 have  been previously computed to obtain $\alpha_k$. Nevertheless, as we will discuss in our next section, the uniformly random
 choice $\theta_k \simeq U([1/2,3/2])$ is an  efficient practical option, which  clearly represents a heuristic proposal.
\end{remark}

Since $\widehat{D}_k$ is a gradient-type descent direction, for our third acceleration strategy, we will adapt
the  steplength associated with the recently developed ABBmin low-cost gradient method \cite{frasso, zhou}, which has proved to be very effective in the
 solution of general nonlinear unconstrained optimization problems \cite{frasso, serafino}. It is worth mentioning that the ABBmin method is a  nonmonotone
scheme for which convergence to local minimizers has been established \cite{serafino}. As in the case of the randomly relaxed acceleration,
this approach can also be viewed as a member of the family for which the acceleration is generated by the method itself. For our specific algorithms, the
third strategy is obtained by substituting the optimal steplength $\alpha_k$ by
\begin{equation}  \label{abbmin}
 \widehat{\alpha}_k = \left\{ \begin{array}{llc}
                         \min\{\alpha_j^{BB2}\: :\: j=\max\{1,k-M\},\dots,k\}, & \;\mbox{ if } &   \alpha_k^{BB2}/ \alpha_k^{BB1} < \tau; \\ [4mm]
                         \alpha_k^{BB1}, &   \mbox{ otherwise } &
                         \end{array}
                         \right. \\ [2mm]
\end{equation}
where $\tau \in (0,1)$ (in practice $\tau\approx 0.8$), $M$ is a small nonnegative integer, and the involved parameters are given by
\[ \alpha_k^{BB1} = \frac{\|S^{(k-1)}\|_F^2}{\langle S^{(k-1)}, Y^{(k-1)}\rangle} \;\; \mbox{ and }\;\;
\alpha_j^{BB2} = \frac{\langle S^{(j-1)}, Y^{(j-1)}\rangle}{\|Y^{(j-1)}\|_F^2} \mbox{ for all } 1\leq j\leq k,
\]
where $S^{(j-1)} = X^{(j)} - X^{(j-1)}$ and $Y^{(j-1)} = \widehat{D}_j - \widehat{D}_{j-1}$, for all $j$.

A geometrical as well as an algebraic motivation for the choice $\widehat{\alpha}_k$ in (\ref{abbmin}) can be found in  \cite{frasso, serafino, zhou}.
 In particular,  they establish an interesting connection  between $\alpha_k^{BB1}$, $\alpha_k^{BB1}$, and the ratio  $\alpha_k^{BB2}/ \alpha_k^{BB1}$,
  with the  eigenvalues (and eigenvectors) of the underlying Hessian of the objective function.  In general, the relationship between the choice of
   steplength, in  gradient-type methods, and the eigenvalues and eigenvectors of the underlying Hessian of the objective function is well-known, and for
   nonmonotone methods can be traced  back to \cite[pp. 117-118]{ghr93}; see also \cite{dasmundi,  RaydanSvaiter}.

\section{Illustrative numerical examples}
\label{numres}

To give further insight into the behavior of the MinCos method for least-squares problems and the three discussed acceleration strategies, we present
the results of some numerical experiments.   All computations were performed in Matlab, using double precision, which has unit roundoff
 $\mu \approx 1.1\times 10^{-16}$. Our initial guess is chosen as $X^{(0)}= \beta I$, where  $\beta>0$ is fixed to satisfy the scaling performed
  at Step 7 in Algorithm Algorithm \ref{mincos} and also in Algorithm \ref{mincosls}, i.e., $\beta=\sqrt{n}/\|A\|_F$
  for Algorithm \ref{mincos} and $\beta=\sqrt{n}/\|A\|^2_F$  for Algorithm \ref{mincosls}.
In every experiment, we stop the process when the merit function $F(X^{(k)})$ (Algorithm \ref{mincos}) or   $\widehat{F}(X^{(k)})$
(Algorithm \ref{mincosls}) is less than or equal to $\epsilon$, for some small  $ \epsilon >0$.
Concerning the package  EPSfun, we use the STEA2 option which has proved to be more effective than STEA1 for our experiments, with different
 choices of the key parameters NCYLCE and MCOL. We notice that when using the STEA2 strategy, the number of reported iterations is given by
 NCYLCE times MCOL. Concerning the ABBmin method, we set $\tau=0.8$ and $M=10$ in all cases. We consider test matrices from the Matlab gallery
 (Poisson2D, Poisson3D, Wathen, Lehmer, and normal), and also from the Matrix Market \cite{MatrixMarket}.

\subsection{PSD matrices using Algorithm \ref{mincos}}

For our first set of experiments we consider symmetric positive definite matrices that are not badly conditioned as the ones obtained
in a 2D or 3D discretization of  Poisson equations with Dirichlet boundary conditions. In Figures \ref{Poisson1}  and \ref{Poisson2} we
report the convergence behavior of the MinCos method (Algorithm \ref{mincos}), the Random Mincos acceleration, and the STEA2 acceleration
 for different values of  NCYLCE and  MCOL, when applied to the gallery matrices Poisson2D with $n=100$  and Poisson3D with $n=1000$, respectively.
  We notice that in both cases the two acceleration schemes need significantly less iterations than the MinCos method to achieve the requires accuracy.
   We also note, in Figure \ref{Poisson1}, that STEA2  outperforms the Random Mincos acceleration. However, we can notice in Figure \ref{Poisson2} that
    when the size of the matrix increases, as well as the condition number, then the Random Mincos acceleration outperforms the STEA2 scheme.

\begin{figure}[!h!]
\begin{center}
    \includegraphics[width=6in,height=2.5in]{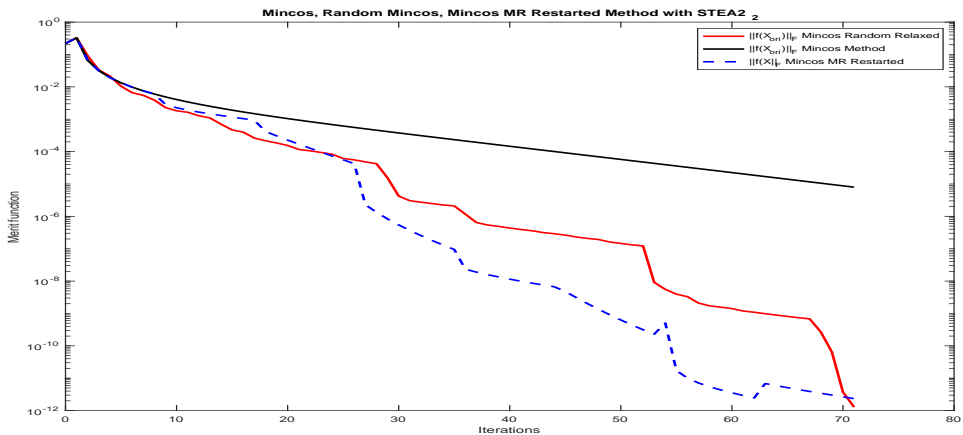}
    \end{center}
\caption{Convergence history of  MinCos, Random Mincos and STEA2 for the 2D Poisson matrix for $n=100$, NCYLCE=8, and MCOL=8.} \label{Poisson1}
\end{figure}

\begin{figure}[!h!]
\begin{center}
    \includegraphics[width=6in,height=2.5in]{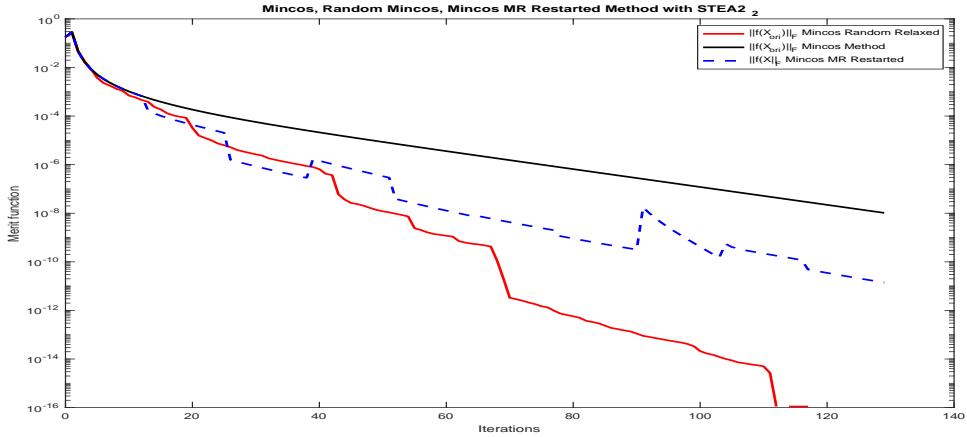}
    \end{center}
\caption{Convergence history of MinCos, Random Mincos and STEA2 for the 3D Poisson matrix for $n=1000$,
NCYLCE=10, and MCOL=12.} \label{Poisson2}
\end{figure}

For the next experiments either the  matrix is sparse and of large size or the matrix is dense. In these cases, we will focus
 our attention on the Random Mincos and the ABBmin acceleration strategies, which are well suited for large problems,
  since they only require the storage of the direction matrix ${D}_k$ and very low additional computational cost per iteration. These results
   are reported in Figure \ref{Poisson3} (for the Poisson 2D matrix with $n=900$)  and in Figure \ref{Lehmer2} (for for the Lehmer matrix with $n=20$).
    We note that the Lehmer matrices, from the Matlab Gallery, are dense.
   We can observe that the ABBmin acceleration represents an aggressive option
   that sometimes outperforms the Random Mincos scheme (for example in Figure \ref{Poisson3}), but it  shows a  highly nonmonotone behavior that
   can produce unstable calculations. The highly  nonmonotone performance of the ABBmin scheme can be noticed in Figure \ref{Lehmer2},
   in which the Random Mincos shows a better acceleration with a numerically trustable monotone behavior.

\begin{figure}[!h!]
\begin{center}
    \includegraphics[width=6in,height=2.5in]{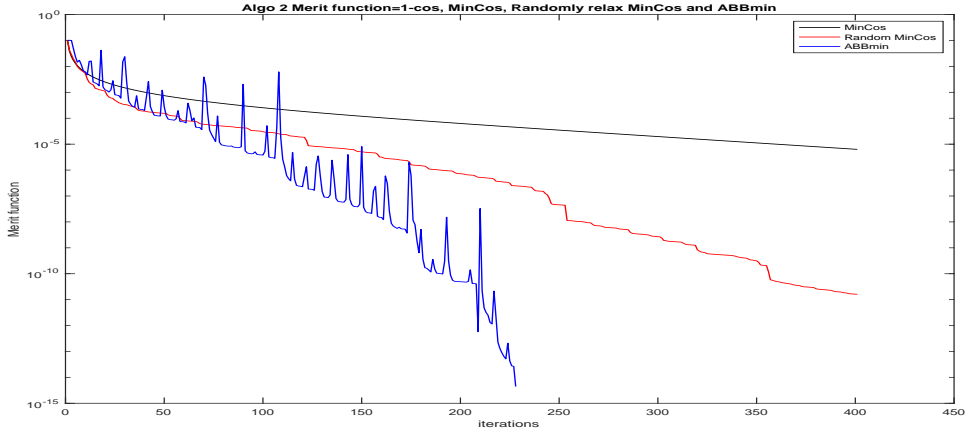}
    \end{center}
\caption{Convergence history of MinCos, Random Mincos and ABBmin for the 2D Poisson matrix for $n=900$  and $\epsilon=10^{-14}$, and
maxiter=400.} \label{Poisson3}
\end{figure}

\begin{figure}[!h!]
\begin{center}
    \includegraphics[width=6in,height=2.5in]{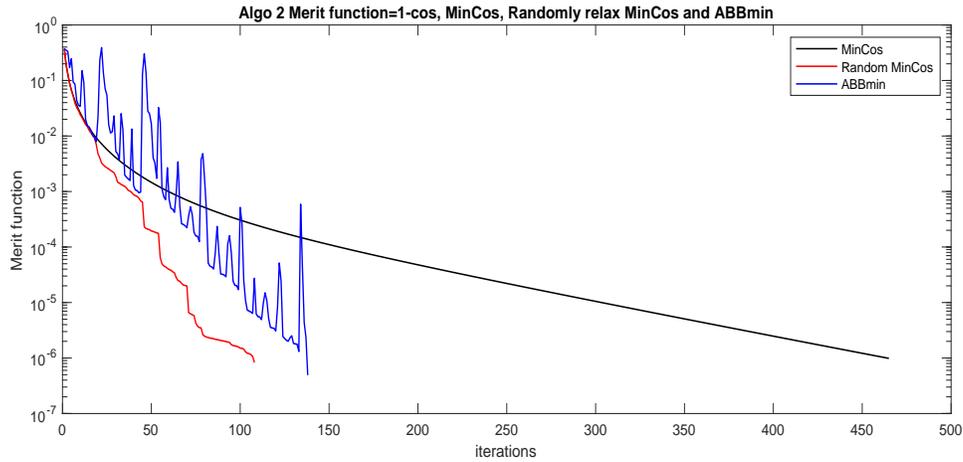}
    \end{center}
\caption{Convergence history of MinCos, Random Mincos and ABBmin for the Lehmer matrix  for $n=20$,  $\epsilon=10^{-6}$, and maxiter=450.} \label{Lehmer2}
\end{figure}

For our next experiments we use the Wathen matrix and the 2D Poisson matrix, both from the Matlab gallery, with different large dimensions,
and we compare the  convergence history of the MinCos method and the Random  Mincos acceleration.  In Figures  \ref{wathen2} and
\ref{wathen3} we can observe the  significant acceleration obtained by the Random Mincos for Wathen(30) of size $n=2821$  ($\epsilon=10^{-6}$,  maxiter=900), and
Wathen(50) of size $n=7701$  ($\epsilon=10^{-8}$,  maxiter=300), respectively. As a consequence we note that
 the Random Mincos scheme exhibits in both cases a clear reduction in the required number of iterations when compared with the MinCos method. Let us recall
 that since the inverse of these matrices is dense, we are dealing with $n^2$ unknowns for all the considered problems, hence in the specific case
  of the wathen matrix ($n=7701$) it is a very large number of variables.
 Similarly, in Figure \ref{Poisson3b} we can notice the clear acceleration and reduction in number of iterations of the Random Mincos acceleration
  for large-scale problems, as compared to the MinCos method.

\begin{figure}[!h!]
\begin{center}
    \includegraphics[width=6in,height=2.5in]{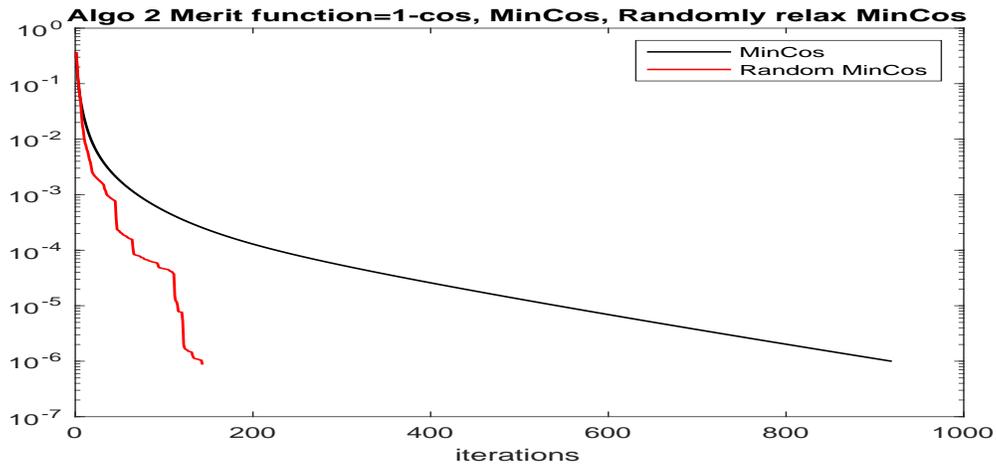}
    \end{center}
\caption{Convergence history of MinCos and the Random MinCos for the Whaten(30) matrix  for $n=2821$ ($3\times 30^2 + 4\times 30 +1$),
 $\epsilon=10^{-6}$, and maxiter=1000.}
\label{wathen2}
\end{figure}

\begin{figure}[!h!]
\begin{center}
    \includegraphics[width=6in,height=2.5in]{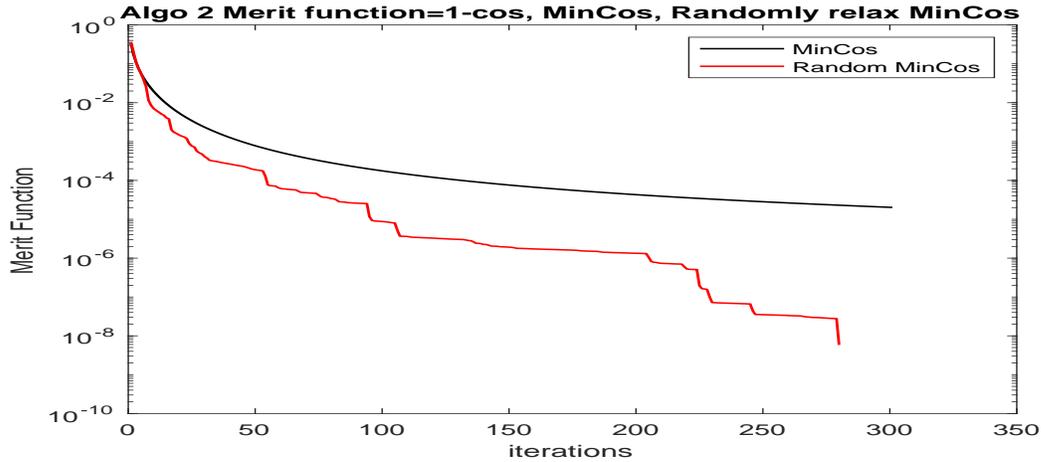}
    \end{center}
\caption{Convergence history of MinCos and the Random MinCos for the Whaten(50) matrix for $n=7701$ ($3\times 50^2 + 4\times 50 +1$),
$\epsilon=10^{-8}$, and maxiter=300.}
\label{wathen3}
\end{figure}

\begin{figure}[!h!]
\begin{center}
    \includegraphics[width=6in,height=2.5in]{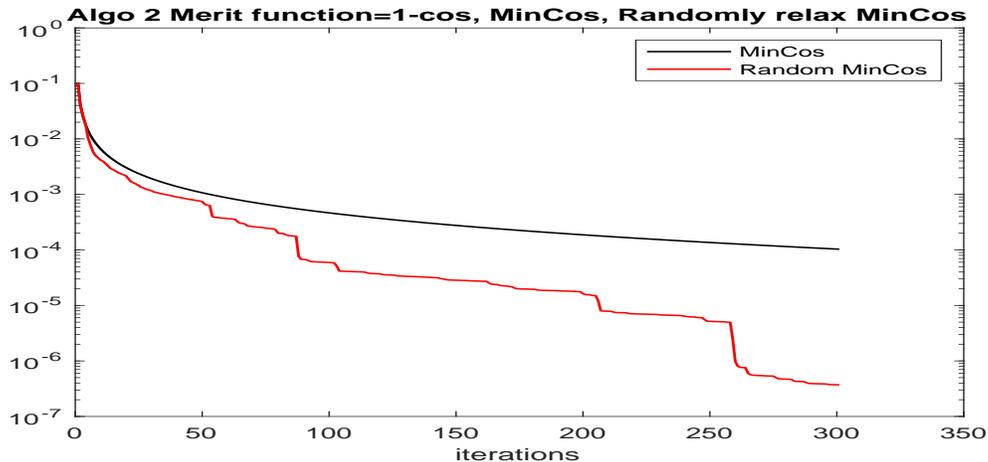}
    \end{center}
\caption{Convergence history of MinCos and the Random MinCos for the 2D Poisson matrix for $n=3969$, $\epsilon=10^{-7}$, and maxiter=300.}
\label{Poisson3b}
\end{figure}

\subsection{Rectangular matrices for Least-square problems using Algorithm \ref{mincosls}}

   We will now consider rectangular matrices $A$ from the Matlab gallery and also from the Matrixmarket \cite{MatrixMarket} (reported in Table \ref{tabexm}),
    and apply Algorithm \ref{mincosls}. We note that in all these experiments, the matrix $A^TA$ is very ill-conditioned.
    In figures \ref{normal1} and \ref{normal2} we report the convergence behavior of the MinCos method (Algorithm \ref{mincosls}),
    the Random Mincos acceleration, and the STEA2 acceleration  for different values of  NCYLCE and  MCOL, when applied to the Matlab random normal  matrices
    (seed=1) for $n=100$ ($m=80$, MAXCOL=8, and NCYCLE=30)  and $n=200$ ($m=160$, MAXCOL=10, and NCYCLE=25), respectively.
    We notice, in Figure \ref{normal1} that  the two acceleration schemes show a much better performance as compared to the MinCos method, requiring
     significantly less iterations to achieve the same accuracy. We can also notice in Figure \ref{normal2} that when the size of the matrix increases,
      as well as the condition number, then the Random Mincos acceleration clearly outperforms the STEA2 scheme.

 \begin{figure}[!h!]
\begin{center}
    \includegraphics[width=6in,height=2.5in]{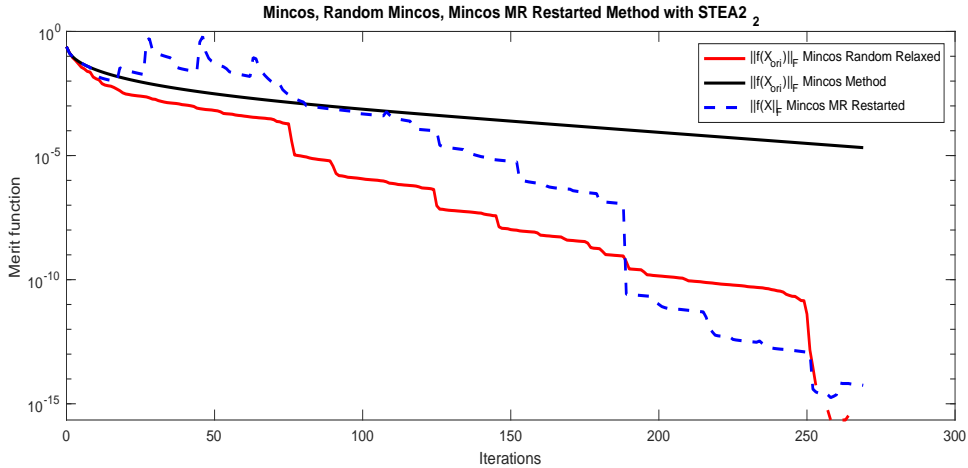}
    \end{center}
\caption{Convergence history of  MinCos, Random Mincos and STEA2 for the random normal matrix for $n=100$, $m=80$, NCYLCE=30, and MCOL=8.} \label{normal1}
\end{figure}

\begin{figure}[!h!]
\begin{center}
    \includegraphics[width=6in,height=2.5in]{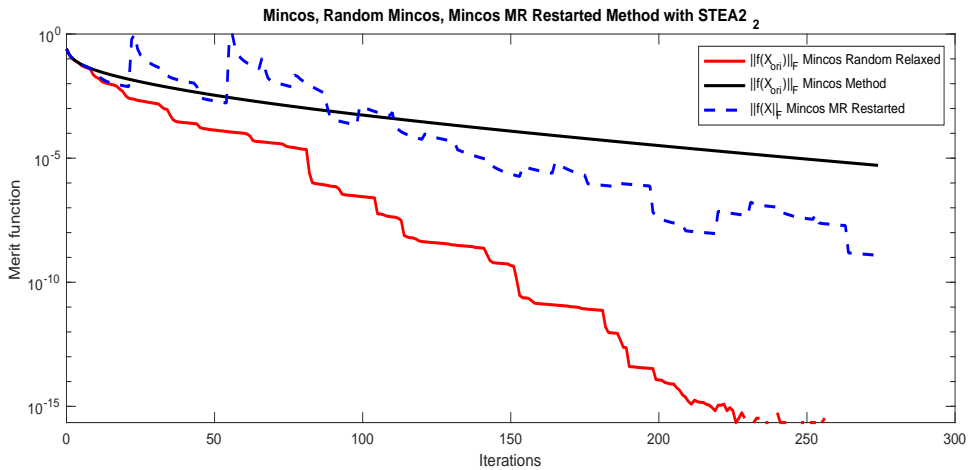}
    \end{center}
\caption{Convergence history of  MinCos, Random Mincos and STEA2 for the random normal matrix for $n=200$, $m=160$, NCYLCE=25, and MCOL=10.} \label{normal2}
\end{figure}

   Next we  compare the performance of the  MinCos method, the Random Mincos acceleration and the ABBmin acceleration on the matrix well1850.
   In Figure \ref{RDM2b} we notice  that the ABBmin scheme shows a similar nonmonotone acceleration as before, up to an accuracy of $10^{-2}$, and after
    that it becomes unstable and cannot reach the required precision. In sharp contrast, the Random Mincos acceleration represents a trustable option that
     clearly outperforms the MinCos method.

   \begin{table}[h]
   \begin{center}
   \begin{tabular}{|l|l|l|}
   \hline
   Matrix & Size $(m,n)$ &  $Cond(A^TA)$\\
   \hline
   illc1850 & (1850,712) &1.4033e+07\\
        \hline
   Well1850 & (1850,712) & 1.2309e+05\\
   \hline
   \end{tabular}
   \caption{Rectangular Matrices form Matrix market} \label{tabexm}
   \end{center}
   \end{table}

   \begin{figure}[h!]
\begin{center}
     \includegraphics[width=6in,height=2.5in]{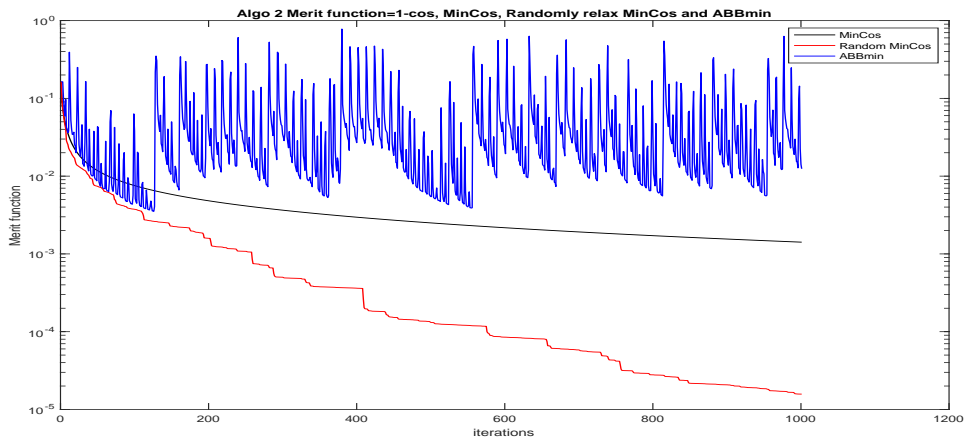}
    \end{center}
\caption{Convergence history of MinCos, the Random MinCos, and ABBmin  for the matrix well1850, with $\epsilon=10^{-5}$ and maxiter=1000.} \label{RDM2b}
\end{figure}

  In Figure \ref{RDM1} we report the convergence history of the MinCos method and the random Mincos acceleration for the harder illc1850 matrix.
   We note that the Random Mincos scheme shows a significant acceleration, and so it needs less iterations than the MinCos method  to achieve
  the same accuracy.

\begin{figure}[h!]
\begin{center}
    \includegraphics[width=6in,height=2.5in]{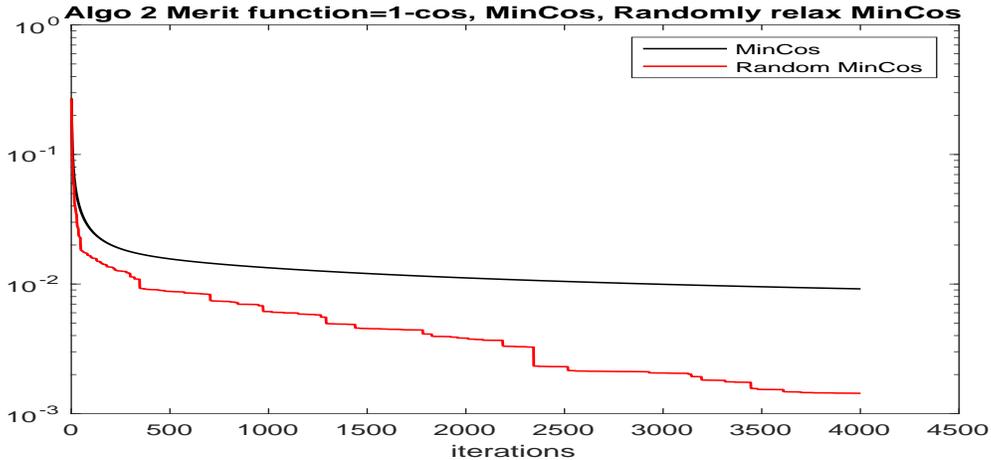}
    \end{center}
\caption{Convergence history of MinCos and the Random MinCos for the matrix illc1850, with $\epsilon=10^{-3}$ and maxiter=4000.} \label{RDM1}
\end{figure}

\newpage
\section{Conclusions}

 We have extended the MinCos iterative method, originally developed in \cite{ChehabRaydan15} for symmetric and positive definite matrices,
 to approximate the inverse of the matrices associated  with linear least-squares problems, and we have also described and adapted three
  different possible acceleration schemes to the generated convergent matrix sequences.

    Our experiments have shown that the geometrical MinCos scheme  is also a robust option to approximate
    the inverse of matrices of the form $A^TA$, associated with least-squares problems, without requiring the explicit knowledge of $A^T$.
    They also show that both schemes (Algorithms \ref{mincos} and \ref{mincosls}) can be significantly accelerated by the three discussed strategies.
    In particular, the STEA2 scheme, from the simplified topological $\varepsilon$-algorithms family, and the Random MinCos scheme are clearly the
    most effective and most stable options. Our conclusion is that for small to medium size problems which are not  ill-conditioned, the STEA2 scheme
     represents a good option with a strong mathematical support. For large-scale and ill-conditioned problems, our conclusion
      is that the inexpensive and numerically trustable Random MinCos acceleration is the method of choice.

      An interesting  application of inverse approximation techniques is the development of preconditioning strategies, for which in many real
      problems a sparse approximation is required.  In that case, a suitable dropping or filtering strategy can be adapted, as it was done and
      extensively discussed in \cite{ChehabRaydan15}. Therefore, the MinCos method has been already combined with dropping strategies showing
      a convenient performance to obtain sparse inverse approximations. Finally, we note that in our results we have compared the different schemes
       using  $\widehat{F}(X) = 1-\cos(X(A^TA),I)$ as the merit function. Similar results can also be obtained using as the merit function the
        norm of the residual, i.e., $\|(I - X(A^TA))\|_F$, as it was already reported for the MinCos method for symmetric and positive definite matrices
         in \cite{ChehabRaydan15}. \\ [4mm]

  \noindent
   {\bf Acknowledgments.} The second author was supported by CNRS  throughout a 3 months  stay {\it Poste Rouge}, in the LAMFA  Laboratory
    (UMR 7352) at  Universit\'e de Picardie Jules Verne, Amiens, France, from february to april, 2019.


\begin{thebibliography}{99}

\bibitem{andraytar} R. Andreani, M. Raydan and P. Tarazaga [2013],
 On the geometrical structure of symmetric matrices,
 {\it Linear Algebra and its Applications},  438,  1201--1214.



\bibitem{Brezinski00} C. Brezinski [2000], Convergence acceleration during the 20th century,
{\it  J. Comput. Appl. Math.}, 122, 1--21. Numerical Analysis in the 20th Century Vol. II: Interpolation and Extrapolation.

\bibitem{BrezinskiZ91} C. Brezinski and M. Redivo-Zaglia [1991], {\it  Extrapolation Methods. Theory and Practice}, North–Holland, Amsterdam.


\bibitem{BrezinskiZ14} C. Brezinski and M. Redivo-Zaglia [2014], The simplified topological $\varepsilon$-algorithms for accelerating sequences
 in a vector space, {\it  SIAM J. Sci. Comput.}, 36, 2227--2247.


  \bibitem{BrezinskiZ17} C. Brezinski and M. Redivo-Zaglia [2017], The simplified topological $\varepsilon$-algorithms: software and applications,
 {\it  Numer. Algor.}, 74, 1237--1260.


 \bibitem{BrezinskiZ18b} C. Brezinski and M. Redivo-Zaglia [2019], The genesis and early developments of Aitken's process, Shanks' transformation,
  the $\varepsilon$-algorithm, and related fixed point methods,  {\it  Numer. Algor.}, 84, 11--133.


\bibitem{BrezinskiZ18} C. Brezinski, M. Redivo-Zaglia, Y. Saad [2018], Shanks sequence transformations and Anderson acceleration,
 {\it  SIAM Rev.}, 60, 646--669.



 \bibitem{carrbg} L. E. Carr,  C. F. Borges, and F. X. Giraldo [2012], An element based spectrally optimized approximate inverse
  preconditioner for the Euler equations, {\it  SIAM J. Sci. Comput. }, 34, 392--420.


\bibitem{Chehab} J.-P. Chehab [2007], Matrix differential equations and inverse preconditioners,
{\it  Computational and Applied Mathematics}, 26, 95--128.


\bibitem{Chehab16} J.-P. Chehab [2016], Sparse approximations of matrix functions via numerical integration of ODEs,
{\it  Bull. Comput. Appl. Math.}, 4, 95--132.



\bibitem{ChehabRaydan} J.P. Chehab and M. Raydan [2008],  Geometrical properties of the Frobenius condition number for positive definite matrices,
 {\it Linear Algebra and its Applications}, 429, 2089--2097.


\bibitem{ChehabRaydan15} J.P. Chehab and M. Raydan [2016], Geometrical inverse preconditioning for symmetric positive definite matrices,
{\it Mathematics},  4(3), 46, doi:10.3390/math4030046.


\bibitem{Chen01} K. Chen [2001], An analysis of sparse approximate inverse preconditioners for boundary
integral equations, {\it SIAM J. Matrix Anal. Appl.}, 22,  1058–-1078.




\bibitem{ChowSaad97} E. Chow and Y. Saad [1997], Approximate inverse techniques for block-partitioned matrices,
 {\it SIAM J. Sci. Comput.}, 18,  1657--1675.



\bibitem{Chung} J. Chung, M. Chung, and D. P. O'Leary [2015], Optimal regularized low rank inverse approximation,
 {\it Linear Algebra and its Applications},  468, 260--269.


  \bibitem{Cui} X. Cui and K. Hayami [2009], Generalized approximate inverse preconditioners for least squares problems,
 {\it Japan J. Indust. Appl. Math.},  26, 1--14.


 \bibitem{dasmundi}  R. De Asmundis, D. di Serafino, F. Riccio, and G. Toraldo [2013], On spectral properties of steepest
  descent methods,   {\it IMA J. Numer. Anal.}, 33, 1416--1435.


 \bibitem{delay80} J. P. Delahaye and B. Germain-Bonne  [1980], R\'esultats n\'egatifs en acc\'el\'eration de la convergence,
 {\it Numer. Math.}, 35, 443--457.


  \bibitem{forsman} K. Forsman, W. Gropp, L. Kettunen, D. Levine, and J. Salonen [1995], Solution of dense systems of linear equations
   arising from integral equation formulations, {\it Antennas and Propagation Magazine}, 37, 96--100.


\bibitem{frasso} G. Frassoldati, L. Zanni, and G. Zanghirati [2008], New adaptive stepsize selections in gradient methods,
  {\it J. Ind. Manag. Optim.}, 4, 299--312.


\bibitem{ghr93} W. Glunt and T. L. Hayden and M. Raydan [1993], Molecular Conformations from Distance Matrices,
  {\it Journal of Computational Chemistry}, 14,  114--120.



\bibitem{gouldscott} N. I. M. Gould and J. A. Scott [1998], Sparse approximate-inverse preconditioners using norm-minimization techniques,
  {\it SIAM J. Sci. Comput.}, 19,  605--625.


 \bibitem{Graves00} P. R. Graves-Morris, P. R. Roberts, and A. Salam [2000], The epsilon algorithm and related topics,
 {\it J. Comput. Appl. Math.}, 122, 51--80.



\bibitem{helsing} J. Helsing [2006], Approximate inverse preconditioners for some large dense random
  electrostatic interaction matrices, {\it BIT Numerical Mathematics}, 46, 307--323.


 \bibitem{jbilou16} K. Jbilou and A. Messaoudi [2016], Block extrapolation methods with applications,
  {\it Applied Numerical Mathematics}, 106,  154--164.



 \bibitem{jbilou00} K. Jbilou and H. Sadok [2000], Vector extrapolation methods: applications and numerical comparison,
  {\it J. Comput. Appl. Math.}, 122,  149--165.


\bibitem{jbilou15} K. Jbilou and H. Sadok [2015], Matrix polynomial and epsilon–-type extrapolation methods with applications,
  {\it Numer. Algor.}, 68,  107--119.


 \bibitem{matos92} A. C. Matos [1992],  Convergence and acceleration properties for the vector $\epsilon$-algorithm,
  {\it Numer. Algor.} 3,  313--320.


 \bibitem{MatrixMarket} Matrix Market, {\tt http://math.nist.gov/MatrixMarket/}


 \bibitem{gonzalez2013} G. Montero, L. Gonz\'alez, E. Fl\'orez, M. D. Garc\'{\i}a, and A. Su\'arez [2002], Approximate inverse computation
 using Frobenius inner product, {\it Numerical Linear Algebra with Applications}, 9, 239--247.


 \bibitem{RaydanSvaiter} M. Raydan and B. Svaiter [2002],  Relaxed steepest descent and Cauchy-Barzilai-Borwein method,
 {\it Computational Optimization and Applications}, 21,  155--167.


  \bibitem{sajo} A. M. Sajo-Castelli, M. A. Fortes, and M. Raydan [2014],  Preconditioned conjugate gradient method for finding
 minimal energy surfaces on Powell-Sabin triangulations, {\it Journal of Computational and Applied Mathematics}, 268, 34--55.


 \bibitem{serafino} D. di Serafino, V. Ruggiero, G. Toraldo, and L. Zanni [2018],  On the steplength selection in gradient methods for
 unconstrained optimization, {\it Applied Mathematics and Computation}, 318, 176--195.


 \bibitem{sidje} R. B. Sidje and Y. Saad [2011], Rational approximation to the Fermi--Dirac function with applications in density functional theory,
  {\it Numer. Algor.}, 56,  455--479.


 \bibitem{wang} J.-K. Wang and S.D. Lin [2014],  Robust inverse covariance estimation under noisy measurements,
 in {\it Proceedings of the 31st International Conference on Machine Learning (ICML-14)},  928--936.



\bibitem{zhou} B. Zhou, L. Gao, and Y. H. Dai [2006],  Gradient methods with adaptive step-sizes,
 {\it Comput. Optim. Appl.}, 35,  69--86.


\end{thebibliography}
\end{document}